%% file: template.tex
\documentclass[11pt]{amsart}
\usepackage[dvipsnames,usenames]{color}
\usepackage{hyperref}
\usepackage{graphicx, import}
\usepackage{epsfig}
\usepackage[latin1]{inputenc}
\usepackage{amsmath}
\usepackage{amsfonts}
\usepackage{amssymb}
\usepackage{amsthm}
\usepackage{amscd}
\usepackage{verbatim}
\usepackage{subfigure}
\usepackage{caption}
\usepackage{pinlabel}
\usepackage{stmaryrd}
\usepackage{enumerate, enumitem}
\usepackage{todonotes}
\usepackage{bm}
\usepackage{thmtools}
\usepackage{thm-restate}
\usepackage{lipsum}
\usepackage{setspace}
\usepackage{mathtools}
\usepackage[all]{xypic}
\usepackage[abs]{overpic}
\usetikzlibrary{arrows.meta, decorations.pathmorphing}
\allowdisplaybreaks

\usepackage{mathdots}
\usepackage{tikz}
\usetikzlibrary{arrows}
\usetikzlibrary{decorations.pathreplacing}
\usepackage{verbatim}
\usetikzlibrary{cd}
\tikzset{taar/.style={double, double equal sign distance, -implies}}
\tikzset{amar/.style={->, dotted}}
\tikzset{dmar/.style={->, dashed}}
\tikzset{aar/.style={->, very thick}}

\usepackage[section]{placeins}

\newtheorem{theorem}{Theorem}[section]

\newtheorem{proposition}[theorem]{Proposition}

\theoremstyle{definition}
\newtheorem{definition}[theorem]{Definition}

\theoremstyle{remark}
\newtheorem{remark}[theorem]{Remark}

\def\F{\mathbb{F}}

\def\R{\mathbb{R}}
\def\Z{\mathbb{Z}}

\newcommand{\Int}{\mathop{\mathrm{int}}\nolimits} 
\newcommand{\cl}{\mathop{\mathrm{cl}}\nolimits}
\newcommand{\Inv}{\mathop{\mathrm{Inv}}\nolimits}

\newcommand{\setof}[1]{\left\{ {#1}\right\}}

\include{correctm}

\showcomments

\author[A. Parikh]{Konstantin Mischaikow and Aakash Parikh}
\email{mischaik@math.rutgers.edu, ap1792@math.rutgers.edu}
\address{Rutgers University, New Brunswick, NJ, USA}
\thanks{K.M. and A.P. were partially supported by the Air Force Office of Scientific Research under awards numbered FA9550-23-1-0011 and FA9550-23-1-0400.}

\numberwithin{equation}{section}

\title{Realizing Saddle-Node Bifurcations from Finite Data}

\begin{document}
\maketitle
\begin{abstract} 
Given a finite set of data generated by an unknown ordinary differential equation it is impossible to exactly determine the associated vector field, and hence, bifurcation theory tells us that it is impossible, in general, to correctly characterize the underlying dynamics.
In this paper, we bypass the effort of obtaining an analytic approximation of the vector field, and we adopt an approach based on Occam's razor: identify the simplest robust characterization of the dynamics that is compatible with the given data. 
Our fundamental assumption is that the data allows for the construction of an isolating block over a parameter space whose homological Conley index is consistent with a saddle-node bifurcation. 
Our main result establishes that, for phase spaces of dimension greater than or equal to $6$, the original vector field can be smoothly deformed into a canonical model exhibiting exactly one structurally stable saddle-node bifurcation. 
Crucially, this deformation leaves the vector field unaltered outside the isolating block, ensuring strict compatibility with the observed data. 
\end{abstract}
%
%

\section{Introduction}
\label{sec:intro}


It is clear that we are entering an era of data driven dynamics. 
What is far from clear is how to proceed from a finite set of data to a generally accepted understanding of the underlying dynamics.
To put this comment into perspective, consider the classical form of scientific modeling via a parameterized system of ordinary differential equations (ODEs)
\begin{equation}
\label{eq:ode}
\frac{dx}{dt} = f(x,\lambda),\quad x\in \R^n,\ \lambda\in\Lambda.   
\end{equation}

There are at least four approaches to characterizing the associated dynamics: (i) identify explicit analytic solutions, (ii) use numerical methods to approximate solutions, (iii) identify invariant measures, and (iv) identify invariant sets.
Each of the approaches has strengths and weaknesses.
However,  successful characterization by any one of these methods provides an accepted understanding of the dynamics.

In this paper we focus on the fourth approach.
In particular, we assume that we are given a finite set of data associated with vector fields  at a finite set of parameter values  and that the process that generated this data can be modeled by an ODE of the form of \eqref{eq:ode}. 
A naive goal, to deduce complete information about the structure of the associated invariant sets, is impossible to achieve.
In general, if $n\geq 2$, then identification of dynamics up to conjugacy requires an uncountable set of invariants and hence cannot be achieved from a finite data set \cite{foreman:rudolph:weiss}.

A less naive goal is to attempt to identify an analytic expression $g$ that approximates the vector field $f$.
Even a brief survey of this approach exceeds the scope of this paper (see for example \cite{brunton:proctor:kutz} and its citations).
There are good reasons for the popularity of this approach; the output is a system of differential equations that can be analyzed and applied using traditional numerical techniques. 

However, there are at least three reasons why alternative/complementary methods are worth pursuing.
First, ensuring that $g$ is a good approximation of $f$ typically requires considerable data.
Second, bifurcation theory tells us that even if $g$ is a good approximation of $f$, the dynamics (on the level of invariant sets) of $g$ need not agree with the dynamics of $f$.
In particular, dynamics can change over Cantor sets of parameters \cite{palis:takens:93}.
Finally, knowledge of $g$ does not imply knowledge of the dynamics induced by $g$; the focus of much of applied dynamics involves identify the dynamics of differential equations.

With this in mind, and in the spirit of Occam's razor, we pursue a weaker objective: 
\begin{quote}
    \emph{Identify the simplest robust characterization of the dynamics that is compatible with the data.}
\end{quote}
  
We motivate our approach via a particularly simple example.
Assume the unknown vector field $f(x,\lambda)$ is defined on $\R$, the parameter space $\Lambda = [0,1]$, and $f$ gives rise to a parametrized family of flows $\Phi\colon \R\times \R \times \Lambda \to \R \times \Lambda$ where $\Phi(t,x,\lambda) = (\varphi_\lambda(t,x),\lambda)\in \R\times \setof{\lambda}$.
Observe that within this formula we are using $\varphi_\lambda \colon \R \times \R \to \R$ to denote the flow associated with the  ODE $\dot{x} = f(x,\lambda)$.
The information that we are given about $f$ comes from a finite sampling of the vector field at different parameters, e.g., $\setof{f(x_n,\lambda_n)}$ as shown in Figure~\ref{fig:data}(a).

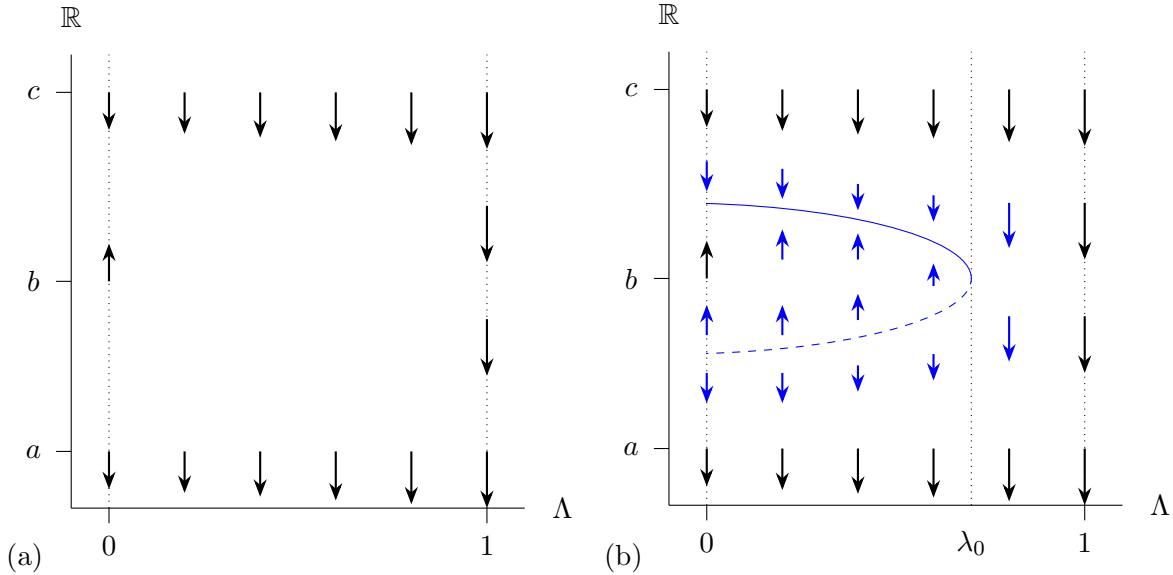
\begin{figure}[hbt]
\begin{picture}(450,250)(0,-10)

\put(0,0){(a)}
\put(0,0){
\begin{tikzpicture}[scale=1]

\draw (0,0) -- (6,0);
\draw (0,0) -- (0,6);

\draw[dotted,thin] (0.5,0) -- (0.5,6);
\draw[dotted,thin] (5.5,0) -- (5.5,6);

\draw(6.5,0) node{$\Lambda$};
\draw(0.5,-0.2) -- (0.5,0);
\draw(5.5,-0.2) -- (5.5,0);
\draw(0.5,-0.5) node{$0$};
\draw(5.5,-0.5) node{$1$};

\draw(0,6.5) node{$\R$};
\draw(-0.2,0.75) -- (0,0.75);
\draw(-0.2,3) -- (0,3);
\draw(-0.2,5.5) -- (0,5.5);
\draw(-0.5,0.75) node{$a$};
\draw(-0.5,3) node{$b$};
\draw(-0.5,5.5) node{$c$};

\foreach \i in {0,...,5}{
\draw[-Stealth,thick] (0.5+\i,0.75) -- ++ (0,-0.5 - .05*\i);
\draw[-Stealth,thick] (0.5+\i,5.5) -- ++ (0,-0.5 - .05*\i);
}

\draw[-Stealth,thick] (0.5,3) -- ++ (0,0.5);

\foreach \i in {0,...,1}{
\draw[-Stealth,thick] (5.5,2.5+1.5*\i) -- ++ (0,-0.75);
}
\end{tikzpicture}
}

\put(225,0){(b)}
\put(225,0){
\begin{tikzpicture}[scale=1]

\draw (0,0) -- (6,0);
\draw (0,0) -- (0,6);

\draw[dotted,thin] (0.5,0) -- (0.5,6);
\draw[dotted,thin] (5.5,0) -- (5.5,6);
\draw[dotted,thin] (4,0) -- (4,6);

\draw(6.5,0) node{$\Lambda$};
\draw(0.5,-0.2) -- (0.5,0);
\draw(5.5,-0.2) -- (5.5,0);
\draw(0.5,-0.5) node{$0$};
\draw(5.5,-0.5) node{$1$};
\draw(4.0,-0.5) node{$\lambda_0$};

\draw(0,6.5) node{$\R$};
\draw(-0.2,0.75) -- (0,0.75);
\draw(-0.2,3) -- (0,3);
\draw(-0.2,5.5) -- (0,5.5);
\draw(-0.5,0.75) node{$a$};
\draw(-0.5,3) node{$b$};
\draw(-0.5,5.5) node{$c$};

\foreach \i in {0,...,5}{
\draw[-Stealth,thick] (0.5+\i,0.75) -- ++ (0,-0.5 - .05*\i);
\draw[-Stealth,thick] (0.5+\i,5.5) -- ++ (0,-0.5 - .05*\i);
}

\draw[-Stealth,thick] (0.5,3) -- ++ (0,0.5);

\draw[-Stealth,blue,thick] (0.5,2.25) -- ++ (0,0.4);
\draw[-Stealth,blue,thick] (0.5,1.75) -- ++ (0,-0.4);
\draw[-Stealth,blue,thick] (0.5,4.55) -- ++ (0,-0.4);

\draw[-Stealth,blue,thick] (1.5,2.25) -- ++ (0,0.4);
\draw[-Stealth,blue,thick] (1.5,1.75) -- ++ (0,-0.4);
\draw[-Stealth,blue,thick] (1.5,3.25) -- ++ (0,0.4);
\draw[-Stealth,blue,thick] (1.5,4.45) -- ++ (0,-0.4);

\draw[-Stealth,blue,thick] (2.5,2.45) -- ++ (0,0.35);
\draw[-Stealth,blue,thick] (2.5,1.85) -- ++ (0,-0.35);
\draw[-Stealth,blue,thick] (2.5,3.25) -- ++ (0,0.35);
\draw[-Stealth,blue,thick] (2.5,4.25) -- ++ (0,-0.35);

\draw[-Stealth,blue,thick] (3.5,2) -- ++ (0,-0.35);
\draw[-Stealth,blue,thick] (3.5,2.9) -- ++ (0,0.3);
\draw[-Stealth,blue,thick] (3.5,4.1) -- ++ (0,-0.35);

\foreach \i in {0,...,1}{
\draw[-Stealth,blue,thick] (4.5,2.5+1.5*\i) -- ++ (0,-0.6);
\draw[-Stealth,thick] (5.5,2.5+1.5*\i) -- ++ (0,-0.75);
}

\begin{scope}
\clip (0.5,4) rectangle (5,3);
\draw[blue] (0,3) ellipse(4 and 1.0);    
\end{scope}

\begin{scope}
\clip (0.5,2) rectangle (5,3);
\draw[blue,dashed] (0,3) ellipse(4 and 1.0);    
\end{scope}

\end{tikzpicture}
}

\end{picture}  
\caption{(a) Arrows indicate sampled vector field $f(x,\lambda)$, $x\in \R$ and $\lambda\in [0,1]\subset \Lambda$.
(b) Black and blue arrows  indicate vector field of $h(x,\lambda)$, $x\in \R$ and $\lambda\in [0,1]\subset \Lambda$.
Black arrows agree with those of (a).
Solid blue curve indicates stable equilibria. 
Dashed blue curve indicates unstable equilibria.}
\label{fig:data}
\end{figure}

Based on  the above mentioned data we restrict our attention to the dynamics on the interval $[a,c]\subset \R$.
Since the parameterized vector field $f$ defined on $[a,c]\times [0,1]$ is unknown,  we are required to make assumptions before proceeding.
We assume that 
\begin{description}
    \item[S1] $f(a,\lambda) <0$  for all $\lambda \in [0,1]$,
    \item[S2] $f(c,\lambda) <0$ for all $\lambda \in [0,1]$,
    \item[S3] $f(x,1) <0$ for all $x\in [a,c]$, and
    \item[S4] $f(b,0) >0$, and
\end{description}

\begin{remark}
\label{rem:justification}
While assumptions {\bf S1} - {\bf S4} allow us to deduce the type of result being offered by this paper, some form of justification for these assumptions is necessary.
Our philosophy is that models for physical problems satisfy reasonable problem dependent constraints.
For example, if we assume that the set of models that we want to consider satisfy a global Lipschitz bound $L$, then for appropriate $L$   {\bf S1} - {\bf S4} follows from the data.
Alternatively, we can assume that our models come from a centered, continuous Gaussian field in which case the data provides a conditional probability that the vector field satisfies  {\bf S1} - {\bf S4}.
In either case a well defined assumption on the class of models to be considered allows one to deduce explicit assumptions that can be used to draw definitive understanding about the potential dynamics.
\end{remark}

Since this example is meant to foreshadow the results of this paper we recall some notation.
Let $X$ be a compact metric space and let $\varphi\colon \R\times X\to X$ be a flow.
A  set $S\subset X$ is an \emph{invariant set} under $\varphi$ if $\varphi(t,S) = S$ for all $t\in \R$.
Given $B\subset X$, we denote the \emph{maximal invariant set in} $B$ \emph{under} $\varphi$  by $\Inv(B,\varphi)\coloneq\{x\in B|\phi(t,x)\in B\textrm{ for all }t\in \mathbb{R}\}$. A subset $L \subset X$ is called a \emph{local section} if there exists $\delta > 0$ such that $L \cdot [-\delta,\delta] := \{\varphi_t(x) \mid x \in L,\ t \in [-\delta,\delta]\}$ is homeomorphic to $L \times [-\delta,\delta]$. A compact set $B \subset X$ is an \emph{isolating block} if $B$ is a compact set with $\Inv(B,\varphi)\subset \Int(B)$ (where $\Int$ denotes interior), and the boundary $\partial B$ decomposes as a union of three closed sets $\partial B = \partial B^{+} \cup \partial B^{-} \cup \partial B^{t}$, where
\begin{enumerate}
    \item $\partial B^{+}$, $\partial B^{-}$ and $\partial B^{t}$ are local sections,
    \item there exists $\delta > 0$ such that $\partial B^{-} \cdot (0,\delta) \cap B = \partial B^{+} \cdot (-\delta,0) \cap B = \emptyset$,
    \item and if $x \in \partial B^{t}$, then there exist $t_1 < 0 < t_2$ such that $\varphi_{[t_1,t_2]}(x) \subset \partial B^{t},$ $ \varphi_{t_1}(x) \in \partial B^{+},$ and $ \varphi_{t_2}(x) \in \partial B^{-}.$

\end{enumerate}
We call $\partial B^+$ the entrance set and $\partial B^-$ the exit set.

Recall that the $\omega$ and $\alpha$ limits of a set $X_0\subset X$ under $\varphi$ are defined by
\[
\omega(X_0,\varphi):=\bigcap_{s\in [0,\infty)} \cl(\varphi([s,\infty),X_0))\quad \textrm{and}\quad \alpha(X_0,\varphi):=\bigcap_{s\in (-\infty,0]} \cl(\varphi((-\infty,s],X_0))
\] 
where $\cl$ denotes closure.

If $S$ is an isolated invariant set an \emph{attractor} is an invariant set $A\subset S$ which admits a neighborhood $A\subset U\subset S$ such that $\omega(U)=A$. Given an attractor $A\subset S$ we define the \emph{dual repeller} of $A$ by $A^*\coloneq \{x\in S|\omega(x)\cap A=\emptyset\}.$

Returning to the example, observe that {\bf S1} and {\bf S2} imply that $[a,c]$ is an isolating neighborhood for $\varphi_\lambda$ for all $\lambda\in [0,1]$.
Furthermore, {\bf S3} implies that $\Inv([a,c],\varphi_1) = \emptyset$.
{\bf S4} implies that $\Inv([a,b],\varphi_0) \neq \emptyset$ and $\Inv([b,c],\varphi_0)\neq \emptyset$.
Since we are working with flows defined on $\R$, $\Inv([a,b],\varphi_0)$ and $\Inv([b,c],\varphi_0)$ must consist of fixed points and connecting orbits between the fixed points.
Furthermore, $\omega(b,\varphi_0) \subset \Inv([b,c],\varphi_0)$ and $\alpha(b,\varphi_0) \subset \Inv([a,b],\varphi_0)$.
Without further assumptions on $f$ this is essentially the extent to which a rigorous description of the dynamics associated with $\Phi$ restricted to $[a,c]\times [0,1]\subset \R\times \Lambda$ can be provided.

While correct, the description provided above of the dynamics of $\Phi$ is complicated and unenlightening.
This is due to the fact that we are attempting to describe dynamics for all possible parameterized families of flows that match the data and the assumptions.
For example, the fact that  there is no a priori limit on the number of fixed points of $\varphi_\lambda$ for any given $\lambda \in [0,1)$ implies that we can only speak about general invariant sets.

As indicated above our goal is to identify the simplest robust dynamics that is consistent with the data and assumptions {\bf S1} - {\bf S4}. 
Figure~\ref{fig:data}(b) provides an answer to this question.

To be more precise, consider a two parameter system of ODEs
\begin{equation}
\label{eq:homotopyODE}
\frac{dx}{dt} = F(x,\lambda,\sigma),\quad x\in \R,\ \lambda\in [0,1],\ \sigma \in [0,1]   
\end{equation}

such that
\begin{description}
    \item[h1] $F(a,\lambda,\sigma) = f(a,\lambda)$ for all $\lambda\in [0,1]$ and all $\sigma\in[0,1]$,
    \item[h2] $F(c,\lambda,\sigma) = f(c,\lambda)$ for all $\lambda\in [0,1]$ and all $\sigma\in[0,1]$,
    \item[h3] $F(x,1,\mu) < 0$ for all $x\in [a,c]$  and all $\sigma\in[0,1]$, and
    \item[h4] $F(b,0,\sigma) = f(b,0)$ for all $\sigma\in[0,1]$.
\end{description}
For simplicity of notation set
\begin{equation}
\label{eq:Occamode}
\frac{dx}{dt} = h(x,\lambda) = F(x,\lambda,1),\quad x\in \R,\ \lambda\in\Lambda,
\end{equation}
and let  $\Psi\colon \R \times \R \times [0,1] \to \R \times [0,1]$ where $\Psi(t,x,\lambda) = (\psi_\lambda(t,x),\lambda) \in \R\times \setof{\lambda}$ denotes the parameterized flow for \eqref{eq:Occamode}. 

We claim that there exists $F$ such that the dynamics of $\Psi$ satisfies
\begin{description}
    \item[C1] for $\lambda \in [0,\lambda_0)$, $\Inv([a,c],\psi_\lambda)$ consists of a unique hyperbolic stable equilibrium $A_\lambda$, a unique  hyperbolic unstable equilibrium $R_\lambda$, and a unique heteroclinic orbit from $R_\lambda$ to $A_\lambda$;
    \item[C2] a saddle-node bifurcation occurs at $\lambda = \lambda_0$; and
    \item[C3] for $\lambda \in (\lambda_0,1]$, $\Inv([a,c],\psi_\lambda)=\emptyset$.
\end{description}
In addition, $A_0 \in (b,c)$ and $A^*_0 \in (a,b)$.

Observe that {\bf h1} - {\bf h4} ensures not only that $f$ and $h$ agree on the data, but that one can get from the true model $f$ to the idealized model $h$ by a continuous change of parameters that does not contradict the experimental data.
{\bf C1} - {\bf C3} gives precision to the  dynamics shown in Figure~\ref{fig:data}(b).

Bifurcation theory guarantees that if we perturb $h$, then the dynamics described by {\bf C1} - {\bf C3} will be preserved (though for the perturbed system the parameter value $\lambda_0$ where the saddle node occurs may change).
We hope the reader agrees that the description of the dynamics via {\bf C1} - {\bf C3} is more illuminative than that provided for $f$.

The philosophy of our approach is that given the data and reasonable assumptions on the class of models that we wish to entertain for the sake of enlightenment we choose as simple a model as possible.
In principle this simple model could be used to suggest what additional data should be collected to support or contradict the conclusions of the simple model.

For $x\in \R$, proving the existence of $F$ is rather straightforward.
The goal of this paper is to prove an analogous result (see Theorem~\ref{saddlenodethm}) for $x\in \R^n$ with $n\ge 6$.
The challenge in higher dimensional cases is that though the behavior of the unknown vector field $f$ on the boundary of $B$ can be quite complicated we need to show that it can be deformed to a rather simple vector field $h$.

Recall that a pseudo-isotopy of a smooth compact $C^\infty$ manifold $M$ is a diffeomorphism  of $M\times [0,1]\to M\times [0,1]$ which restricts to the identity on $(M\times\{0\})\cup (\partial M\times[0,1])$.
Let $\mathcal{P}$ denote the group of pseudo-isotopies of $M$ where the group operation is composition, and equip $\mathcal{P}$ with the $C^\infty$ topology. The pseudo isotopy theorem states that if $M$ is simply connected and $\textrm{dim}(M)\ge 5$, then $\mathcal{P}$ is connected.

To access the pseudo isotopy theorem, via a homotopy we recast our problem into the setting of Morse theory (see Proposition~\ref{reineckparameterized}).
In particular, this allows us to focus on $\mathcal{G}$ the space of smooth functions $g:M\times [0,1]\to [0,1]$ with $g^{-1}(i)=M\times\{i\}$ for $i=0,1$ and no critical points on the boundary, and $\mathcal{E}\subset \mathcal{G}$, the functions which have no critical points (such as projection to the interval).
Then the pseudo-isotopy theorem may be restated as follows: A path in $\mathcal{G}$ with endpoints in $\mathcal{E}$ can be deformed relative to its endpoints to a path in $\mathcal{E}$.
This puts us into the setting of Cerf theory \cite{Cerf1}, which we exploit to prove Theorem~\ref{saddlenodethm}, the main result of this paper.

\subsection*{Organization} This paper is organized as follows. In Section \ref{sec:background} we recall background material on Conley index theory and some results from Cerf theory. In Section \ref{proof} we state and prove our result.

\subsection*{Acknowledgments} We thank Francois Laudenbach and Filippos Sytilidis for helpful conversations. 

\section{Background}
\label{sec:background}
Consider a parameterized system of ODEs
\begin{equation}
\label{eq:ode2}
\frac{dx}{dt} = f(x,\lambda),\quad x\in 
\R^n,\ \lambda\in \Lambda 
\end{equation}
and the associated parameterized family of flows
\begin{align*}
    \Phi\colon \R\times \R^n\times [0,1] & \to \R^n \times \Lambda \\
    (t,x,\lambda) & \mapsto \Phi(t,x,\lambda) = (\phi_\lambda(t,x),\lambda)
\end{align*}
\begin{remark}
We will exclusively use $\Lambda=[0,1]$ as our space of parameters in what follows, but we stick with the notation $\Lambda$ in order to avoid confusion with other instances of $[0,1]$.
\end{remark}

Given an isolating block $B\subset \R^n\times \Lambda$ for $\Phi$, set $B_\lambda := B\cap \R^n \times \setof{\lambda}$ and $S_\lambda\coloneq \Inv(B_\lambda,\phi_\lambda)$. We say that the isolated invariant set $\Sigma=\Inv(B,\Phi)$ is a continuation between $S_0$ and $S_1$. Say that $A\cup A^*\cup C(A,A^*)$ is a decomposition of $S$ into an attractor $A$ and dual repeller $A^*$ along with the set of connection orbits $C(A,A^*)$.

\begin{definition}\label{simplenbhd}

We say that $B_0$ is a \emph{simple isolating block} if $B_0=M\times[0,1]$ for some manifold $M$ (see the Cerf theory section for why we take $B_0$ to be a product with an interval), $B_0$ is simply connected, and there exist isolating blocks $B_{A^*}, B_A\subset B_0$ for $A^*$ and $A$ respectively such that $B_{A^*}\cup B_A=B_0$,  $B_{A^*}\cap B_A=\partial B_{A^*}^-\cap \partial B_A^+$, and $B_{A^*}$ and $B_A$ along with their entrance and exit sets are simply connected.
\end{definition}
With this notation in place, we can state the definitions needed to formulate our main theorem.
\begin{definition}
\label{defn:SNisolating}

We say that $B\subset \R^n\times \Lambda$ \emph{isolates a saddle-node bifurcation} for \eqref{eq:ode2} if $B$ is an isolating block for $\Phi$ such that:
\begin{enumerate}
\item[(i)]  $B$ is smoothly isotopic to $B_0\times \Lambda\subset \R^n\times\Lambda$.
\item[(ii)]  $B_0$ is a simple isolating neighborhood as in Definition \ref{simplenbhd}.
\item[(iii)] There exists $\lambda_0 \in (0,1)$ such that for all $\lambda \in [0,\lambda_0)$, $\Inv(B_\lambda,\varphi_\lambda) = A_\lambda \cup R_\lambda \cup C(R_\lambda,A_\lambda)$ where 
\begin{itemize}
    \item $A_\lambda$ is a hyperbolic equilibrium with $k-1$ dimensional unstable manifold,
    \item $R_\lambda$ is a hyperbolic equilibrium with $k$ dimensional unstable manifold
    \item $C(R_\lambda,A_\lambda)$ consists of a unique heteroclinic orbit from $R_\lambda$ to $A_\lambda$.
\end{itemize}
\item[(iv)] At $\lambda = \lambda_0$ there is a unique equilibrium that undergoes a saddle-node bifurcation with respect to $\lambda$ (see \cite[Section 8.2]{chicone}, and also Remark \ref{whitneyremark} below).
\item[(v)] For all $\lambda \in (\lambda_0,1]$, $\Inv(B_\lambda,\varphi_\lambda) = \emptyset$. 
\end{enumerate}

\end{definition}
A pair of compact sets $(N,L)$ with
$L \subset N \subset X$ is called an \emph{index pair} for $S$ if
$\operatorname{Inv}(\cl(N \setminus L)) = S$, if whenever $x \in L$ and
$\varphi_{[0,t]}(x) \subset N$ then $\varphi_{[0,t]}(x) \subset L$, and if for
every $x \in N$ such that $\varphi_t(x) \notin N$ for some $t>0$, there exists
$0 \leq s \leq t$ with $\varphi_s(x) \in L$. The \emph{Conley index} of $S$ is
the pointed homotopy type $h(S) := [\,N/L, [L]\,]$, where $(N,L)$ is any index pair
and $L$ is collapsed to a basepoint. The \emph{homology Conley index} is $CH_*(S):=H_*(h(S))$. The Conley index is well defined because there is a flow defined homotopy equivalence between any two index pairs for the same isolated invariant set. If $B$ is an isolating block for $S$ with
boundary decomposition $\partial B = \partial B^{+} \cup \partial B^{-} \cup \partial B^{t}$,
then $(N,L) = (B,\partial B^{-})$ is an index pair for $S$. See \cite{MR511133, ConleyEaston, Salamon1} for many more details about this theory.
\begin{definition}
\label{defn:SNisolating}
We say that an isolating neighborhood for $\Phi$, $B$, \emph{isolates a homological saddle-node bifurcation} for \eqref{eq:ode2} if
\begin{enumerate}
\item[(i)]  $B$ is smoothly isotopic to $B_0\times \Lambda\subset \R^n\times\Lambda$.
\item[(ii)] $B_0$ is a simple isolating neighborhood.
\item[(iii)] $CH_*(\Inv(B_0,\varphi_0)) = 0$.
\item[(iv)] $\Inv(B_0,\varphi_0)$ has a decomposition into an attractor and dual repellor $(A_0,A_0^*)$ such that
\[
CH_m(A_0) \cong \begin{cases}
    \Z & \text{if $m=k-1$}\\
    0 & \text{otherwise,}
\end{cases}
\quad\text{and}\quad
CH_m(A_0^*) \cong \begin{cases}
    \Z & \text{if $m=k$}\\
    0 & \text{otherwise.}
\end{cases}.
\]

\end{enumerate}
\end{definition}

The main result of this paper is as follows.

\begin{theorem}\label{saddlenodethm}
Assume that $B$ isolates a homological saddle-node bifurcation for \eqref{eq:ode2} with $n\ge 6$.
Then, there exists a smooth parameterized family of ODEs
\begin{equation}\label{2parameterODE}
    \frac{dx}{dt}  = F(x,\lambda,\sigma), \quad x\in \mathbb{R}^n,\ \lambda \in \Lambda,\ \sigma\in [0,1]
\end{equation}
\end{theorem}
such that 
\begin{enumerate}
    \item[(i)] $F(x,\lambda,\sigma) = f(x,\lambda)$ for all $(x,\lambda) \in (\R^n \times [0,1])\setminus \Int(B)$ and all $\sigma\in [0,1]$, 
    \item[(ii)] $F(x,\lambda,0) = f(x,\lambda)$ for all $(x,\lambda)\in \R^n \times [0,1]$, 
    \item[(iii)] $B$ isolates a saddle-node bifurcation for $\frac{dx}{dt} = F(x,\lambda,1)$.
\end{enumerate}

 \begin{remark}\label{whitneyremark}
    In the context of a gradient dynamical system, (iii) in Theorem \ref{saddlenodethm} can be reformualated as saying that $F(x,\lambda,1)$, in some system of coordinates $x=(y,z)\in \R^{n-1}\times \R$ in a neighborhood of $\lambda_0$, is given by $\nabla g$ where \begin{equation}\label{whitcoord}g(x,\lambda)=g(x_0,\lambda_0)+z^3\pm(\lambda- \lambda_0) z+Q(y)\end{equation} and $Q$ is a non-degenerate quadratic form on $\R^{n-1}$. These coordinates are the so-called \emph{Whitney coordinates}.
 \end{remark}

\subsection{Cerf theory}\label{Cerfsection}

In this section we will consider a number of function spaces defined as subsets of $C^\infty(X,Y)$ for compact manifolds with boundary $X$ and $Y$; for us, all such functions spaces carry the weak Whitney $C^\infty$ topology. For what follows, let $X=M\times [0,1]$ and let $X_i=M\times{i}$ where $M$ is a simply connected manifold with boundary. 
\begin{definition}
Let $\mathcal{G}$ be the space of $C^\infty$ functions $g:X\to [0,1]$ such that $g^{-1}(i)=X_i$ for $i=0,1$ and $g$ has no critical points on $\partial X$. 
\end{definition}

\begin{definition}  
A function $g\in\mathcal{G}$ has a \emph{cubic singularity} at the point $x\in X$ if the Hessian of $g$ at $x$ has nullity 1.
\end{definition}
Following Cerf, we now introduce the first two levels of the stratification of $\mathcal{G}$ \cite{Cerf1}. A Morse function is \emph{excellent} if no two critical points share a critical value. The set of excellent Morse functions is denoted $\mathcal{G}^0$ and is the \emph{zeroeth stratum} in the stratification of $\mathcal{G}$. A function $g\in \mathcal{G}$ is a \emph{generalized Morse function} if $g$ is Morse away from one  cubic singularity and no two critical points share a critical value. The set of generalized Morse functions is denoted $\mathcal{G}^1_\alpha$. A function $g\in\mathcal{G}$ is a \emph{crossing Morse function} if $g$ is Morse and exactly two critical points of $g$ share a critical value. The set of crossing Morse functions is denoted $\mathcal{G}^1_\beta$. Setting $\mathcal{G}^1:=\mathcal{G}^1_\alpha\cup\mathcal{G}^1_{\beta}$, $\mathcal{G}^1$ is the \emph{first stratum} in the stratification of $\mathcal{G}.$ 

Cerf theory establishes results about how a generic path of functions in $\mathcal{G}$ is positioned with respect to the stratification just introduced. These statements should be thought of as analogous to the classical result that generic smooth functions are Morse.

Let $g(x,\lambda)$ be a smooth path in $\mathcal{G}$ with $\lambda$ the path parameter. Then we say that $g(x,\lambda)$ has an \emph{isolated cubic singularity at} $(x_0,\lambda_0)$ if the Hessian of $g(x_0,\lambda_0)$ is corank $1$, $g(x,\lambda_0)|_{X\backslash\{x_0\}}$ is Morse, and there exists $\epsilon>0$ such that $g(x,\lambda)$ is Morse for $0<|\lambda-\lambda_0|<\epsilon$. The following proposition due to Whitney determines the local form of an isolated cubic singularity in a smooth path\cite{MR98418}. 
\begin{proposition}\label{whitney}
    If $g(x,\lambda)\in \textrm{Maps}([0,1],\mathcal{G})$ has an isolated cubic singularity at $(x_0, \lambda_0)$, then there is a coordinate transformation $x=(y,z)\in \R^{n-1}\times \R$ and a $\delta>0$ such that for all $\lambda\in(\lambda_0-\delta,\lambda_0+\delta)$, $g(x,\lambda)$ is given by Equation \ref{whitcoord}.
\end{proposition}
With the definition of an isolated cubic singularity in hand, we may define a Cerf path, which should be thought of as the analog for paths of what a Morse function is to a smooth function.
\begin{definition}\label{cerfpath}
    A \emph{Cerf path} is a smooth path $g(x,\lambda)$ in $\mathcal{G}^0\cup \mathcal{G}^1\subset \mathcal{G}$ which intersects $\mathcal{G}^1:=\mathcal{G}^1_\alpha\cup \mathcal{G}^1_\beta$ at finitely many parameter values $\lambda_1,...,\lambda_k$ such that for $i=1,...,k$, $g(x,\lambda_i-\epsilon)$ and $g(x,\lambda_i+\epsilon)$ belong to different components of $\mathcal{G}^0$ for sufficiently small $\epsilon$.
\end{definition}
The following theorem is due to Cerf, and is the analog for smooth paths of the theorem that Morse functions are generic in the space of smooth functions.
\begin{theorem}\cite[~$\S$ 3, p.24]{Cerf1}\label{opendense}
    The set of Cerf paths is open and dense in the space of smooth paths in $\mathcal{G}$.
\end{theorem}

Given a smooth path $g$ and any $\lambda\in \Lambda$, let $\textrm{Crit}_g(\lambda)$ be the set of critical values of $g(x,\lambda)$. 
\begin{definition}
     The \emph{Cerf graphic} of $g$ is \[\bigcup_{\lambda\in \Lambda} (\lambda, \textrm{Crit}_g(\lambda))\subset \Lambda\times [0,1].\] 
\end{definition}

As a consequence of the Whitney coordinates, the Cerf graphic of an isolated cubic singularity is a cusp at $\lambda_0$, as seen in the case of a cubic birth singularity in Figure \ref{fig:cusp}.
\begin{figure}
    \centering
    \includegraphics[width=.15\linewidth]{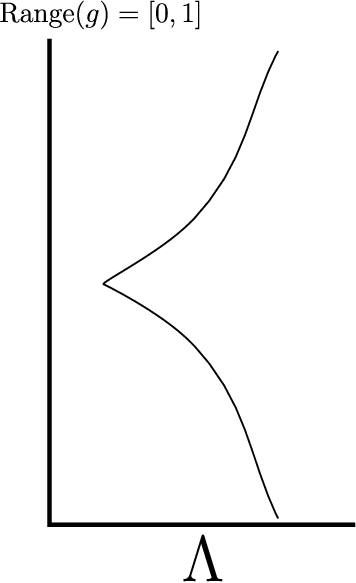}
    \caption{A cubic birth singularity in a Cerf graphic.}
    \label{fig:cusp}
\end{figure}
\begin{theorem}[Pseudo-isotopy theorem]\label{pseudoisotopy}
Assume that $X$ is simply connected and $\dim(X)\ge 5$. If $g(x,\lambda)$ is a smooth path which has a Cerf graphic that is empty at $\lambda=0$ and $\lambda=1$, then there exists a smooth two parameter family $g(x,\lambda,\sigma)$ extending $g(x,\lambda)$ rel $g(x,0)$ and $g(x,1)$ such that $g(x,\lambda,1)$ is a path with empty Cerf graphic.
\end{theorem}
An example of the pseudo-isotopy theorem is given in Figure \ref{pseudoisotopyfig}.

\begin{figure}
    \centering
    \includegraphics[width=0.5\linewidth]{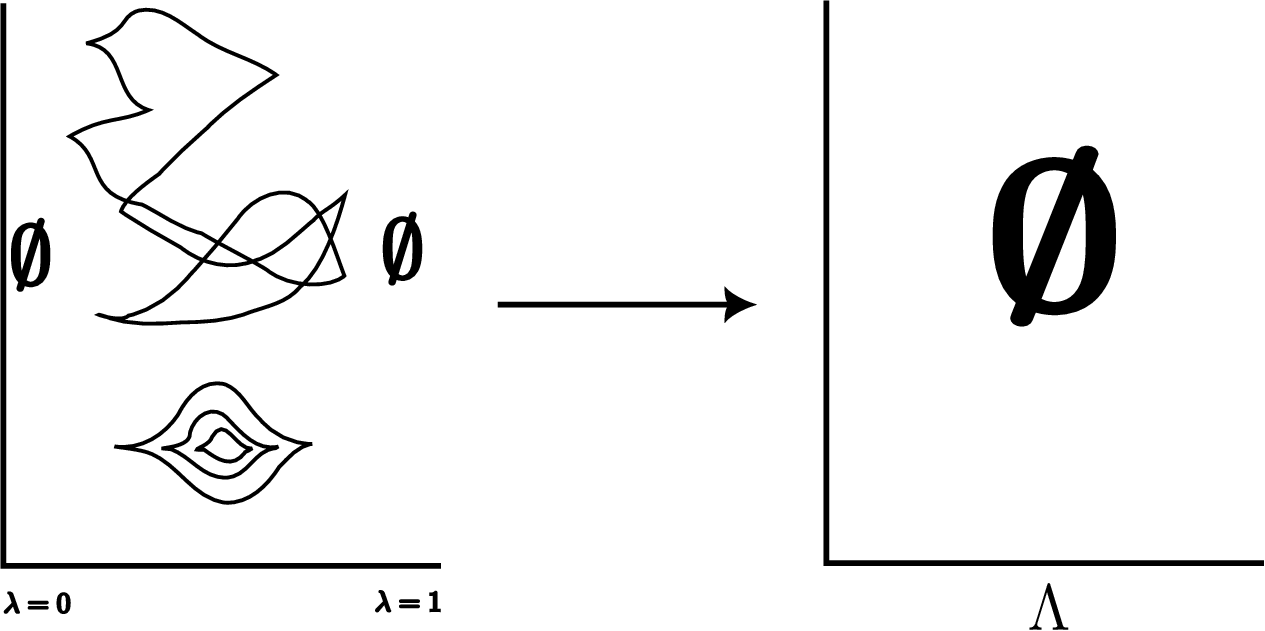}
    \caption{Pictured above are two \emph{Cerf graphics:} diagrams displaying the critical points of a smooth family of generalized Morse functions $g:M\times [0,1]\to \R$ on a manifold $M$. Left: A Cerf graphic which is empty at parameter values $\lambda=0$ and $\lambda=1.$ Right: A Cerf graphic which is empty at all parameter values $\lambda$. The pseudo-isotopy theorem of Jean Cerf states that if the $\dim(M)\ge 5$ and $M$ is simply connected, there is a two parameter family of smooth functions interpolating the two Cerf graphics.}
    \label{pseudoisotopyfig}
\end{figure}

The last result that we need from Cerf theory is  \emph{uniqueness of birth} \cite[~Chapter 3, $\S$1]{Cerf1}.

\begin{proposition}[Uniqueness of birth]\label{unique}
Assume that $\dim(X)\ge 6$. If $g(x,\lambda)$ is a path of Morse functions with Cerf graphic consisting of two disjoint arcs of critical points $p_\lambda$ and $q_\lambda$ such that $p_\lambda$ and $q_\lambda$ are algebraically canceling for some $\lambda_0\in \Lambda$, then there is a two parameter family $g(x,\lambda,\sigma)$ for $\sigma\in[0,1]$ extending $g(x,\lambda)$ rel $g(x,0)$ and $g(x,1)$ such that the Cerf graphic of $g(x,\lambda,1)$ consists of a cubic death singularity followed by a cubic birth singularity.
\end{proposition}

Uniqueness of birth is illustrated in Figure \ref{fig:uniquenessbirth}. Technically we are using Smale's cancellation lemma in conjunction with uniqueness of birth in the above proposition to allow for algebraically instead of geometrically canceling critical points in its hypotheses. 
\begin{figure}
    \centering
    \includegraphics[width=0.5\linewidth]{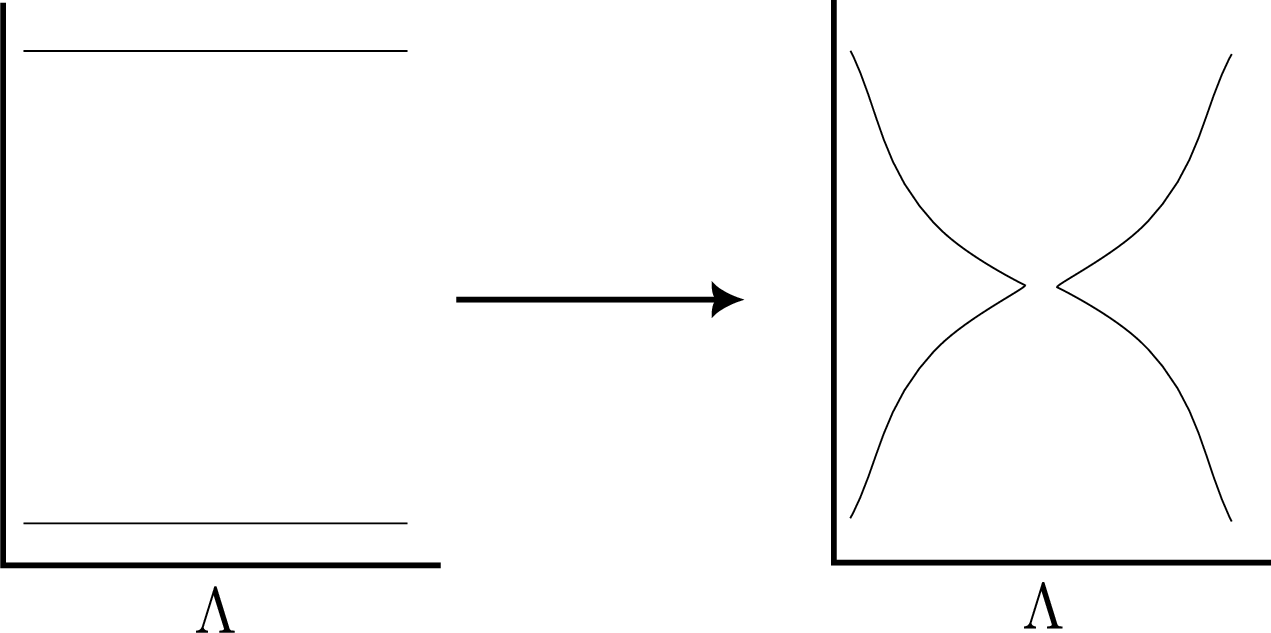}
    \caption{Uniqueness of birth}
    \label{fig:uniquenessbirth}
\end{figure}

\section{Result}\label{proof}
In Proposition~\ref{reineckparameterized} below, we prove a one-parameter generalization of the following theorem due to J.\ Reineck .
\begin{theorem}\cite{Reineck2}\label{reineck2}
Let $B$ be an isolating block for $\frac{dx}{dt} = f(x)$, $x\in \R^n$.
Then, there exists a one parameter family of ODEs 
\begin{equation}
\label{eq:R2}
\frac{dx}{dt} = F(x,\sigma), \quad\sigma\in [0,1],   
\end{equation} 
such that 
\begin{enumerate}
    \item $F(x,0) = f(x)$
    \item $B$ is an isolating neighborhood for \eqref{eq:R2} for all $\sigma\in [0,1]$,
    \item there exists $B'\subset \Int(B)$ such that $\Inv(B') = \Inv(B)$ for all $\sigma\in [0,1]$, and 
    \item there exists a Morse function $g\colon B \to \R$ such that $F(x,1)\big|_{B'} = -\nabla g(x)$.
\end{enumerate}
\end{theorem}

\begin{proposition}\label{reineckparameterized}
Consider a parameterized system of ODEs $\frac{dx}{dt}=f(x,\lambda)$, where $x\in \R^n$ and $\lambda \in \Lambda$. Assume that $B$ is an isolating block for all $\lambda\in \Lambda$.
Then, there is a two parameter family of ODEs \begin{equation}\frac{dx}{dt}=F(x,\lambda,\sigma),\quad\sigma\in [0,1]\end{equation} such that
\begin{enumerate}
    \item $F(x,\lambda,0)=f(x,\lambda)$
    \item $B$ is an isolating neigbhorhood for all $\lambda\in \Lambda$, $\sigma\in[0,1]$
    \item there exists $B'\subset \Int(B)$ such that $\Inv(B'_\lambda)=\Inv(B_\lambda)$ for all $(\lambda,\sigma)\in\Lambda\times [0,1]$ and
    \item there is a Cerf path $g(x,\lambda)$ such that $F(x,\lambda,1)\Big|_{B'}=-\nabla g(x,\lambda)$.
\end{enumerate}
\end{proposition}
As mentioned already, the proof given below is essentially identical to the proof of Reineck's Theorem \ref{reineck2} except for the necessary modifications to start with a parameterized system of ODEs instead of the one ODE.
\begin{proof}
Say that the maximal invariant set of $B$ is $S$. 

Let $g(x,\lambda)$ be a smooth Lyapunov function for $S$ defined on a neighborhood $U$ of $B$, meaning that for $(x,\lambda)\in U\backslash S$, \[0>\langle \nabla g(x,\lambda), f(x,\lambda)\rangle.\] The existence of such a Lyapunov function is established in \cite[~Lemma 5.3]{Salamon1}.

    Let $B'$ be a compact neighborhood of $S$ such that $cl(B')\subset \textrm{int}(B)$ and let $\rho:\R^n\times[0,1]\to[0,1]$ be a smooth cutoff function satisfying
    \[\rho(x,\lambda)=\begin{cases}1&\textrm{ for }(x,\lambda)\in B'\\0&\textrm{ for }(x,\lambda)\in \R^n\backslash B\end{cases}.\] 
    Consider the two parameter family of vector fields $R(x,\lambda,\sigma)$ defined by \begin{equation}\label{vecfam}R(x,\lambda,\sigma)\coloneqq\rho(x,\lambda)[(1-\sigma)f(x,\lambda)-\sigma\nabla g(x,\lambda)]+(1-\rho(x,\lambda))f(x,\lambda).\end{equation} For any $\sigma,\lambda\in [0,1]$, $g$ is decreasing along orbits of $R(x,\sigma,\lambda)$ in $U\backslash S$ since
    \begin{equation}\label{decrease}
    \begin{split}
   \langle \nabla g(x,\lambda), f(x,\lambda,\sigma)\rangle&=\langle\nabla g(x,\lambda),\rho(x,\lambda)[(1-\sigma)f(x,\lambda)-\sigma\nabla g(x,\lambda)]+(1-\rho(x,\lambda))f(x,\lambda)\rangle\\&=(1-\sigma)\rho(x,\lambda)\langle \nabla g(x,\lambda), f(x,\lambda)\rangle-\sigma\rho(x,\lambda)\langle \nabla g(x,\lambda),\nabla g(x,\lambda)\rangle\\ 
    &+(1-\rho(x,\lambda))\langle\nabla g(x,\lambda),f(x,\lambda)\rangle<0.
    \end{split}
    \end{equation}

From Equation \ref{decrease}, we can see that $B$ is an isolating block for each $\sigma\in[0,1]$. Indeed, if $(x,\lambda)\in B\backslash S$ and $g(x,\lambda)\le 0$, follow the orbit of $(x,\lambda)$ along $\frac{dx}{dt}=R(x,\lambda,\sigma)$. Since $g(x,\lambda)$ is strictly decreasing along the orbit and all the critical points of $g(x,\lambda)$ are contained in the set $\{(x,\lambda)|g(x,\lambda)=0\}$, the orbit must leave $B$ in forward time due to compactness. Similarly, if $(x,\lambda)\in B\backslash S$ and $g(x,\lambda)\ge0$, the backward orbit of $(x,\lambda)$ must leave $B$. Hence, we see that $B$ is an isolating block for each $\sigma$. We identically conclude that $B'$ is an isolating neighborhood for each $\sigma$ and that the maximal invariant sets of $B$ and $B'$ agree for each $\sigma$. From Equation \ref{vecfam} and the definition of the cutoff function $\rho$ it is clear that we have achieved $(1),(2)$ and $(3)$ in the Proposition with $F=R$. In addition, $R(x,\lambda,1)|_{B'}=\nabla g(x,\lambda)$ so $(4)$ also holds with the caveat that $g(x,\lambda)$ is not necessarily a Cerf path yet.

By \cite[~Theorem 1.7]{ConleyEaston} there is an open set $V$ in the compact open topology about $\nabla g(x,\lambda)$ in the space of vector fields on $\mathbb{R}^n\times[0,1]$ such that for $\varphi\in U$, $B$ is an isolating block for $\varphi$. By Proposition \ref{opendense} Cerf functions are generic, and so we see that there exists a small perturbation $g(x,\lambda,\sigma)$ for $\sigma\in[0,1]$ so that $g(x,\lambda,1)$ is a Cerf path and $B$ is an isolating block for $\nabla g(x,\lambda,\sigma)$. If we define \[F(x,\lambda,\sigma)=\begin{cases}
    R(x,\lambda,2\sigma)&\textrm{for }0\le \sigma\le \frac{1}{2}\\-\rho(x,\lambda)\nabla g(x,\lambda,2\sigma-1)+(1-\rho(x,\lambda))f(x,\lambda)&\textrm{for }\frac{1}{2}\le \sigma\le 1
\end{cases}\]
then $(4)$ holds as well.
\end{proof}

With the aid of this technical proposition, we can now prove Theorem \ref{saddlenodethm}. We will assume that $B=B_0\times\Lambda$ throughout instead of $B$ smoothly isotopic to $B_0\times\Lambda$ since the isotopy only obscures notation. Therefore $B_\lambda=B_0\times\{\lambda\}$. 

\begin{proof}
   Recall that $B_0$ is a simple isolating neighborhood for the attractor-repellor pair $(A_0,A^*_0)$. Therefore, from Theorem \cite[~Theorem 2.2]{Reineck1} applied to the isolating blocks $B_A,B_{A^*}\subset B_0$, there is a one parameter family $f_0(x,\sigma)$ of vector fields that is constant with respect to $\sigma$ on $\R^n\backslash B_0$ and an isolating neighborhood $B_0'\subset B_0$ with the same invariant set such that $f_0(x,1)\Big|_{B'_0}=-\nabla g_0(x)$ where $g_0$ is a Morse function on $B'_0$ with two algebraically canceling critical points $p_0$ and $q_0$. One more application of Theorem \cite[~Theorem 2.2]{Reineck1} produces a one parameter family of vector fields $f_1(x,\sigma)$ that is constant with respect to $\sigma$ on $\R^n\backslash B_1$ and an isolating neighborhood $B_1'\subset B_1$ with the same invariant set such that $f_1(x,1)\Big|_{B_1'}=-\nabla g_1(x)$ for $g_1$ a Morse function with no critical points. Choose smooth cutoff functions $\eta,\xi\colon [0,1]\to[0,1]$ such that $\eta(\lambda)=1$ near  $\lambda=0$, $\eta(\lambda)=0$ away from $\lambda=0$ and $\xi(\lambda)=0$ away from $\lambda=1$, $\xi(\lambda)=1$ near $\lambda=1.$ Define a two-parameter family of vector fields by\[F_1(x,\lambda,\sigma)=f(x,\lambda)+\eta(\lambda)\bigl(f_0(x,\sigma)-f(x,0)\bigr)+\xi(\lambda)\bigl(f_1(x,\sigma)-f(x,1)\bigr).\]Since the endpoint extensions agree with the original endpoint vector fields near \(\partial B\), the correction terms vanish on a collar of \(\partial B\). 
    
    Making use of Proposition \ref{reineckparameterized}, there exists a two parameter family of vector fields $F_2(x,\lambda,\sigma)$ and an isolating block $B'\subset \textrm{int}((B_0'\cap B_1')\times\Lambda)$ with the same invariant set such that $F_2(x,\lambda,0)=F_1(x,\lambda,1)$ and $F_2(x,\lambda,1)\Big|_{B'}=-\nabla g(x,\lambda)$ for a Cerf path $g(x,\lambda)$. By Smale's cancellation lemma, we may perturb $g(x,\lambda)$ so that $p_0$ and $q_0$ are geometrically canceling. Then an application of Proposition \ref{unique} followed by an application of \ref{pseudoisotopy} allows us to conclude that there is a two parameter family $g(x,\lambda,\sigma)$ for $\sigma\in[0,1]$ such that $g(x,\lambda,0)=g(x,\lambda)$ and $g(x,\lambda,1)$ is a Cerf path whose graphic consists of $p_0$ and $q_0$ merging together in a cubic death singularity; see Figure \ref{check}. Define $F_3(x,\lambda,\sigma)$ so that $F_3(x,\lambda,0)=F_2(x,\lambda,1)$, $F_3(x,\lambda,\sigma)$ is constant with respect to $\sigma$ outside of $B$, and $F_3(x,\lambda,\sigma)|_{B'}=-\nabla g(x,\lambda,\sigma)$.

    \begin{figure}
    \centering
    \includegraphics[width=1\linewidth]{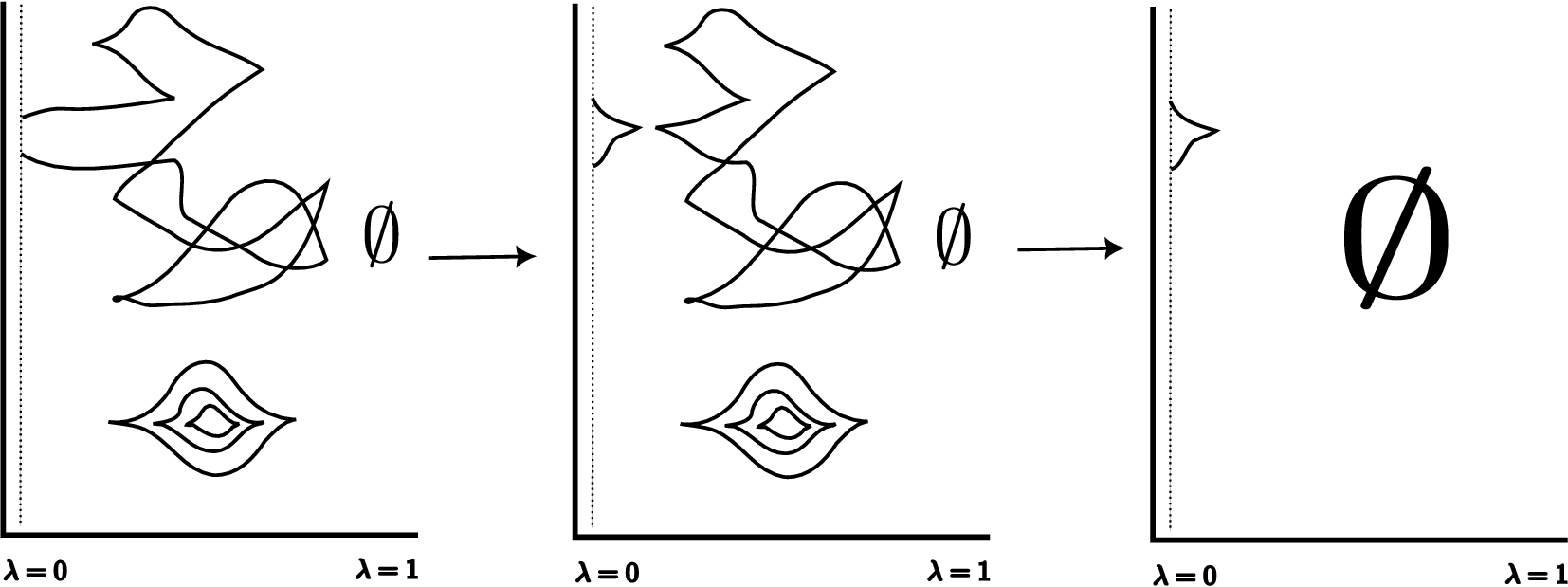}
    \caption{Cerf graphics in the proof of Theorem \ref{saddlenodethm}.}
    \label{check}
\end{figure}
Finally, we can define $F(x,\lambda,\sigma)=\begin{cases}F_1(x,\lambda,3\sigma)&\textrm{for } 0\le \sigma\le\frac{1}{3}\\F_2(x,\lambda,3\sigma-1)&\textrm{for }\frac{1}{3}\le \sigma\le \frac{2}{3}\\F_3(x,\lambda,3\sigma-2)&\textrm{for }\frac{2}{3}\le \sigma\le 1\end{cases}.$

Since taking the gradient of the Whitney coordinates for a cubic death singularity (c.f Equation \ref{whitcoord}) yields a family of vector fields exhibiting a saddle node bifurcation, we see that $B$ isolates a saddle node bifurcation for $\frac{dx}{dt}=F(x,\sigma,1)$. 
\end{proof}

\section{Concluding remarks}

The result of this paper of course begs the question of what happens in lower dimensions. We can prove the same theorem in dimension $n=2$ using the Jordan-Schoenflies theorem in place of Reineck's theorem -- the restriction of what we call being a simple isolating block for $n=2$ guarantees that we are considering disks -- and then following the rest of our proof given above identically (crucially, the pseudo-isotopy theorem holds in dimension $2$). The question of what happens in dimensions $n=3$ and $4$ seems to be out of reach using the methods considered here.

It is of course natural to consider analogous questions about other bifurcation phenomena from the persepective of methods of this paper. While this involves the more delicate questions about the higher homotopy groups of the pseudo-isotopy space $\mathcal{P}(M)$, we believe that at least a partial answer can be given using similar methods for other bifurcation phenomena; we plan to return to this question in a future paper.

Finally, we mention that we believe it is possible to extend the result of this paper to isolating blocks which are not isotopic to products, as much of the machinery of Cerf theory applies to general cobordisms.

\bibliographystyle{amsalpha}
\bibliography{bib}
\end{document}

%% file: correctm.tex
\def\corcommstyle{\bf\small\tt}

\def\corrl #1<<#2||#3>>{
\if\visiblecomments y
  \begin{quote} {\corcommstyle $<<$COMMENT$>>$ {\color{red}#1\marginpar{!!}}\\$<<$OLD$<<$} \end{quote}

{\color{red} 
 #2
 }

  \begin{quote} {\corcommstyle ==NEW== } \end{quote}
   \noindent\hrulefill
 
\vspace{-10pt} 
 
 \noindent\hrulefill
 
 \vspace{-10pt} 
 
 \noindent\dotfill
 
  #3
  
   \noindent\dotfill 

\vspace{-10pt} 
 
 \noindent\hrulefill
 
 \vspace{-10pt} 
 
 \noindent\hrulefill
  \begin{quote} {\corcommstyle $>>$END$>>$ } \end{quote}
 \else
  #3
 \fi
}

\long\def\longcorrl #1<<#2||#3>>{
\if\visiblecomments y
  \begin{quote} {\corcommstyle $<<$COMMENT$>>$ {\color{red}#1\marginpar{!!}}\\$<<$OLD$<<$} \end{quote}
 
 {\color{red}

  #2
  
  }
  
  \begin{quote} {\corcommstyle ==NEW== } \end{quote}
  
    \noindent\hrulefill
 
\vspace{-10pt} 
 
 \noindent\hrulefill
 
 \vspace{-10pt} 
 
 \noindent\dotfill
 
  #3
  
   \noindent\dotfill 

\vspace{-10pt} 
 
 \noindent\hrulefill
 
 \vspace{-10pt} 
 
 \noindent\hrulefill
  \begin{quote} {\corcommstyle $>>$END$>>$ } \end{quote}
 \else
  #3
 \fi
}

\def\mlabel #1
{
  \if\visiblecomments y
     \marginpar[\flushright \bf \footnotesize #1]{\bf \footnotesize #1}
  \fi
}

\def\flabel #1
{
  \if\visiblecomments y
       \hbox{\bf\footnotesize #1}
  \fi
}

\def\corrq #1<<#2>>{
\if\visiblecomments y
  \begin{quote} {\corcommstyle $<<$COMMENT$>>$ {\color{red}#1}\marginpar{!!}\\$<<$BEG$<<$} \end{quote}
  \noindent\hrulefill
 
\vspace{-10pt} 
 
 \noindent\hrulefill
 
 \vspace{-10pt} 
 
 \noindent\dotfill

  #2
 
  \noindent\dotfill 

\vspace{-10pt} 
 
 \noindent\hrulefill
 
 \vspace{-10pt} 
 
 \noindent\hrulefill 
  \begin{quote} {\corcommstyle $>>$END$>>$ } \end{quote}
 \else
  #2
 \fi
}

\long\def\longcorrq #1<<#2>>{
\if\visiblecomments y
  \begin{quote} {\corcommstyle $<<$COMMENT$>>$ #1\marginpar{!!}\\$<<$BEG$<<$} \end{quote}
  \noindent\hrulefill
 
\vspace{-10pt} 
 
 \noindent\hrulefill
 
 \vspace{-10pt} 
 
 \noindent\dotfill

  #2

  \noindent\dotfill 

\vspace{-10pt} 
 
 \noindent\hrulefill
 
 \vspace{-10pt} 
 
 \noindent\hrulefill 
  \begin{quote} {\corcommstyle $>>$END$>>$ } \end{quote}
 \else
  #2
 \fi
}

\def\corrc #1<<>>{
\if\visiblecomments y
  \begin{quote} {\corcommstyle $<<$COMMENT$>>$ \color{red} #1\marginpar{!!}} \end{quote}
\fi
}

\def\corre #1<<#2||#3>>{
\if\visiblecomments y
  #3\marginpar{\corcommstyle #1}
 \else
  #3
 \fi
}

\long\def\longcorre #1<<#2||#3>>{
\if\visiblecomments y
  #3\marginpar{\corcommstyle #1}
 \else
  #3
 \fi
}

\def\corrs #1<<#2||#3>>{
\if\visiblecomments y
  #3\marginpar{\corcommstyle #2 $\rightarrow$ #3\\ #1}
 \else
  #3
 \fi
}

\def\corro #1<<#2||#3>>{
#2}

\def\corrn #1<<#2||#3>>{
#3}

\long\def\longcorro #1<<#2||#3>>{
#2}

\long\def\longcorrn #1<<#2||#3>>{
#3}

\long\def\underconstruction #1<<<#2>>>{
\if\visiblecomments y
  \begin{quote} {\corcommstyle $<<$UNDER CONSTRUCTION - BEGIN$>>$ #1\marginpar{!!}} \end{quote}
  #2
  \begin{quote} {\corcommstyle $>>$UNDER CONSTRUCTION - END$>>$ } \end{quote}
 \else
 \fi
}

\def\showcomments{
  \let\visiblecomments y
}

\def\hidecomments{
  \let\visiblecomments n
}


%% file: bib.bib
@book {chicone,
    AUTHOR = {Chicone, Carmen},
     TITLE = {Ordinary differential equations with applications},
    SERIES = {Texts in Applied Mathematics},
    VOLUME = {34},
   EDITION = {Third},
 PUBLISHER = {Springer, Cham},
      YEAR = {[2024] \copyright 2024},
     PAGES = {xxii+729},
      ISBN = {978-3-031-51651-1; 978-3-031-51652-8},
   MRCLASS = {34-01 (37-01)},
  MRNUMBER = {4769729},
       DOI = {10.1007/978-3-031-51652-8},
       URL = {https://doi.org/10.1007/978-3-031-51652-8},
}

@book {MR511133,
    AUTHOR = {Conley, Charles},
     TITLE = {Isolated invariant sets and the {M}orse index},
    SERIES = {CBMS Regional Conference Series in Mathematics},
    VOLUME = {38},
 PUBLISHER = {American Mathematical Society, Providence, RI},
      YEAR = {1978},
     PAGES = {iii+89},
      ISBN = {0-8218-1688-8},
   MRCLASS = {58Fxx (34-02 35-02)},
  MRNUMBER = {511133},
MRREVIEWER = {Joel\ Smoller},
}

@article {ConleyEaston,
    AUTHOR = {Conley, C. and Easton, R.},
     TITLE = {Isolated invariant sets and isolating blocks},
   JOURNAL = {Trans. Amer. Math. Soc.},
  FJOURNAL = {Transactions of the American Mathematical Society},
    VOLUME = {158},
      YEAR = {1971},
     PAGES = {35--61},
      ISSN = {0002-9947,1088-6850},
   MRCLASS = {57.48},
  MRNUMBER = {279830},
MRREVIEWER = {F.\ W.\ Wilson, Jr.},
       DOI = {10.2307/1995770},
       URL = {https://doi.org/10.2307/1995770},
}

@article { foreman:rudolph:weiss,
    AUTHOR = {Foreman, Matthew and Rudolph, Daniel J. and Weiss, Benjamin},
     TITLE = {The conjugacy problem in ergodic theory},
   JOURNAL = {Ann. of Math. (2)},
  FJOURNAL = {Annals of Mathematics. Second Series},
    VOLUME = {173},
      YEAR = {2011},
    NUMBER = {3},
     PAGES = {1529--1586},
      ISSN = {0003-486X},
   MRCLASS = {37A05 (03E15 28D05 37A35)},
  MRNUMBER = {2800720},
MRREVIEWER = {Mariusz Lema\'{n}czyk},
       DOI = {10.4007/annals.2011.173.3.7},
       URL = {https://doi-org.proxy.libraries.rutgers.edu/10.4007/annals.2011.173.3.7},
}

@article{ brunton:proctor:kutz,
Author = {Brunton, Steven L. and Proctor, Joshua L. and Kutz, J. Nathan},
Title = {Discovering governing equations from data by sparse identification of
   nonlinear dynamical systems},
Journal = {PROCEEDINGS OF THE NATIONAL ACADEMY OF SCIENCES OF THE UNITED STATES OF
   AMERICA},
Year = {2016},
Volume = {113},
Number = {15},
Pages = {3932-3937},
Month = {APR 12},
DOI = {10.1073/pnas.1517384113},
ISSN = {0027-8424},
EISSN = {1091-6490},
ResearcherID-Numbers = {Brunton, Steve/AAG-3871-2019
   Brunton, Steven/AAG-3871-2019},
ORCID-Numbers = {Brunton, Steve/0000-0002-6565-5118
   },
Unique-ID = {WOS:000373762400028},
}

@book { palis:takens:93,
    AUTHOR = {Palis, Jacob and Takens, Floris},
     TITLE = {Hyperbolicity and sensitive chaotic dynamics at homoclinic
              bifurcations},
    SERIES = {Cambridge Studies in Advanced Mathematics},
    VOLUME = {35},
      NOTE = {Fractal dimensions and infinitely many attractors},
 PUBLISHER = {Cambridge University Press},
   ADDRESS = {Cambridge},
      YEAR = {1993},
     PAGES = {x+234},
      ISBN = {0-521-39064-8},
   MRCLASS = {58F14 (58F13 58F15)},
  MRNUMBER = {MR1237641 (94h:58129)},
MRREVIEWER = {Michael Hurley},
}

@article {Salamon1,
    AUTHOR = {Salamon, Dietmar},
     TITLE = {Connected simple systems and the {C}onley index of isolated
              invariant sets},
   JOURNAL = {Trans. Amer. Math. Soc.},
  FJOURNAL = {Transactions of the American Mathematical Society},
    VOLUME = {291},
      YEAR = {1985},
    NUMBER = {1},
     PAGES = {1--41},
      ISSN = {0002-9947,1088-6850},
   MRCLASS = {58F25 (34C35)},
  MRNUMBER = {797044},
       DOI = {10.2307/1999893},
       URL = {https://doi.org/10.2307/1999893},
}

@article {Reineck1,
    AUTHOR = {Reineck, James F.},
     TITLE = {Continuation to the minimal number of critical points in
              gradient flows},
   JOURNAL = {Duke Math. J.},
  FJOURNAL = {Duke Mathematical Journal},
    VOLUME = {68},
      YEAR = {1992},
    NUMBER = {1},
     PAGES = {185--194},
      ISSN = {0012-7094,1547-7398},
   MRCLASS = {58E05 (57R70 58F09 58F25)},
  MRNUMBER = {1185822},
MRREVIEWER = {Christopher\ K.\ McCord},
       DOI = {10.1215/S0012-7094-92-06807-4},
       URL = {https://doi.org/10.1215/S0012-7094-92-06807-4},
}

@article {Reineck2,
    AUTHOR = {Reineck, James F.},
     TITLE = {Continuation to gradient flows},
   JOURNAL = {Duke Math. J.},
  FJOURNAL = {Duke Mathematical Journal},
    VOLUME = {64},
      YEAR = {1991},
    NUMBER = {2},
     PAGES = {261--269},
      ISSN = {0012-7094,1547-7398},
   MRCLASS = {58F09 (58F25)},
  MRNUMBER = {1136376},
MRREVIEWER = {Louis\ Funar},
       DOI = {10.1215/S0012-7094-91-06413-6},
       URL = {https://doi.org/10.1215/S0012-7094-91-06413-6},
}

@article {Cerf1,
    AUTHOR = {Cerf, Jean},
     TITLE = {La stratification naturelle des espaces de fonctions
              diff\'erentiables r\'eelles et le th\'eor\`eme de la
              pseudo-isotopie},
   JOURNAL = {Inst. Hautes \'Etudes Sci. Publ. Math.},
  FJOURNAL = {Institut des Hautes \'Etudes Scientifiques. Publications
              Math\'ematiques},
    NUMBER = {39},
      YEAR = {1970},
     PAGES = {5--173},
      ISSN = {0073-8301,1618-1913},
   MRCLASS = {58D99 (57D35)},
  MRNUMBER = {292089},
MRREVIEWER = {John\ Willard\ Milnor},
       URL = {http://www.numdam.org/item?id=PMIHES_1970__39__5_0},
}

@incollection {MR98418,
    AUTHOR = {Whitney, Hassler},
     TITLE = {Singularities of mappings of {E}uclidean spaces},
 BOOKTITLE = {Symposium internacional de topolog\'ia algebraica
              {I}nternational symposium on algebraic topology},
     PAGES = {285--301},
 PUBLISHER = {Universidad Nacional Aut\'onoma de M\'exico and UNESCO,
              M\'exico},
      YEAR = {1958},
   MRCLASS = {53.00 (26.00)},
  MRNUMBER = {98418},
MRREVIEWER = {R.\ Thom},
}
